\documentclass[12pt,reqno]{amsart}

\usepackage{amsmath,amssymb,amsfonts,amsthm,stmaryrd}
\usepackage{accents}
\usepackage{cite}
\usepackage{colonequals}
\usepackage{enumerate,enumitem}
\usepackage{caption}
\usepackage[justification=centering]{caption}
\usepackage{float}
\setlist[enumerate]{label=(\alph*),font=\normalfont}
\usepackage[top=30truemm,bottom=30truemm,left=30truemm,right=30truemm]{geometry}
\usepackage[breaklinks=true,colorlinks=true,linkcolor=blue,citecolor=blue,filecolor=blue,urlcolor=blue]{hyperref}
\hypersetup{linktocpage}
\usepackage{mathrsfs}
\usepackage{mathtools}
\usepackage{tensor}
\usepackage{slashed}
\usepackage[capitalise]{cleveref}

\newtheorem{thm}{Theorem}[section]
\newtheorem*{thm*}{Theorem}
\newtheorem{ex}[thm]{Example}
\newtheorem*{ex*}{Example}
\newtheorem{cor}[thm]{Corollary}
\newtheorem*{cor*}{Corollary}

\theoremstyle{definition}

\theoremstyle{remark}

\newcommand{\N}{\mathrm{N}}

\DeclareMathOperator{\Gal}{Gal}
\DeclareMathOperator{\Li}{Li}

\title{Significant Digits of Primes in Subsets}


\author{Henry Glunz}

\begin{document}
\maketitle
\begin{abstract}

Benford's Law describes the prevalence of small numbers as the leading digits of numbers in many sets of integers. We prove a variant of Benford's law for many positive-density subsets of the primes.  This follows from a more general result over number fields.
\end{abstract}

\section{Introduction and Statement of Results}
In 1881, astronomer Simon Newcomb noticed that the pages in his table of logarithms beginning with the digit $1$ were more worn than others. \cite{N} This phenomenon is surprising compared to the naive assumption that wear should be equidistributed across digits. Frank Benford formalized Newcomb's observation in the 1930s as Benford's Law, which states that entries with leading digit $d\in \{ 1, \cdots , 9 \}$ should appear with probability $\log_{10}(1+d^{-1})$. Some examples of sets consistent with this leading digit pattern include the population statistics of countries, the molecular weights of chemical compounds, election data, and pandemic infection rates. \cite{B}
The set of integers does not conform to Benford's Law, as
\[
\liminf_{x\rightarrow\infty} \frac{1}{x} \sum_{\substack{n\leq x, \\ n\text{ begins} \\ \text{with digit }1}} 1 = \frac{1}{9}
\qquad\textup{and}\qquad 
\limsup_{x\rightarrow\infty} \frac{1}{x} \sum_{\substack{n\leq x, \\ n\text{ begins} \\ \text{with digit }1}} 1 = \frac{5}{9}.
\]
Despite this, the set of integers does conform to Benford's Law logarithmically, namely
$$\lim_{x \rightarrow \infty} \frac{1}{\log x}  \sum_{\substack{n\leq x, \\ n\text{ begins} \\ \text{with digit }d}} \frac{1}{n} = \log_{10}(1+d^{-1}).$$
Similarly, the set of primes does not conform to this law.  In particular, if 
$$\pi_{10,d}(x) \colonequals \# \{ \text{prime }p \leq x : p\text{'s base-10 expansion begins with the digit $d$}\}$$
and $\pi(x)=\#\{p\leq x \colon \textup{$p$ prime}\}$, then
\[
\liminf_{x\rightarrow\infty}\frac{\pi_{10,1}(x)}{\pi(x)}= \frac{1}{5} 
\qquad\textup{and}\qquad 
\limsup_{x\rightarrow\infty} \frac{\pi_{10,1}(x)}{\pi(x)}= \frac{1}{2}.
\]
Nevertheless, Whitney \cite{W} proved that
\[
\lim_{x\to\infty}\frac{1}{\sum_{p\leq x}\frac{1}{p}}\sum_{\substack{\textup{$p\leq x$} \\ \textup{$p$ begins} \\ \textup{with digit $d$}}}\frac{1}{p}=\log_{10}(1+d^{-1}).
\]

We prove a generalization of Whitney's result for subsets of the primes; this result follows from a general theorem that holds over number fields.  Let $K$ be an number field, and let $\N=\N_{K/\mathbb{Q}}$ be the absolute norm of $K$ over $\mathbb{Q}$. Let $\mathcal{P}$ be the set of prime ideals in $K$, and let $\mathcal{A}$ be a subset of $\mathcal{P}$.  
We now define logarithmic density $d(\mathcal{A})$ of $\mathcal{A}$ to be the common value of
\[
\overline{d(\mathcal{A})} = \limsup_{x \rightarrow \infty} \frac{1}{\log\log x} \sum_{\substack{\mathfrak{p}\in \mathcal{A} \\ \N\mathfrak{p} \leq x}} \frac{1}{\N\mathfrak{p}}\qquad\textup{and}\qquad \underline{d(\mathcal{A})} = \liminf_{x \rightarrow \infty} \frac{1}{\log\log x} \sum_{\substack{\mathfrak{p}\in \mathcal{A} \\ \N\mathfrak{p} \leq x}} \frac{1}{\N\mathfrak{p}}
\]
if they equal each other.  Note that by the prime ideal theorem, we have that
\[
\sum_{\mathrm{N}\mathfrak{p}\leq x}\frac{1}{\N\mathfrak{p}}\sim \log\log x.
\]

\begin{thm}
\label{thm:1}
Let $K$ be a number field, let $\mathcal{P}$ be the set of prime ideals of $K$, and let $\mathcal{A} \subseteq \mathcal{P}$ be a subset such that there exist constants $\delta\geq 0$ and $C\in\mathbb{R}$ (both depending at most on $\mathcal{A}$) such that 
\begin{equation}
\label{1.1}
    \sum_{\substack{\mathfrak{p} \in \mathcal{A} \\ \N\mathfrak{p} \leq x}} \frac{1}{\N\mathfrak{p}} = \delta \log\log x+C+o\Big(\frac{1}{\log x} \Big).
\end{equation}
Let $b\geq 2$ be an integer, and let $\mathcal{A}_{b,S}$ be the subset of $\mathcal{A}$ consisting of prime ideals whose norms have a base-$b$ expansion beginning with a string $S$ of $m$  base-$b$ digits $a_1,\ldots,a_m$ with $a_1\neq 0$. Then we have that $$d(\mathcal{A}_{b,S})=\delta \log_b(1+S^{-1}).$$
\end{thm}

In particular, Theorem \ref{thm:1} can be used to study the significant digits of norms of prime ideals occurring in the Chebotarev density theorem.

\begin{cor}
\label{cor:1}
Let $K$ be a number field and $L$ a
normal extension of $K$ with Galois group $G = \Gal (L/K)$. Let $\mathcal{A}$ be the set of unramified primes $\mathfrak{p}$ in $K$ with Artin symbol $\Big( \frac{L/K}{\mathfrak{p}} \Big)=C$. Then we have
$$d(\mathcal{A}_{b,S}) = \frac{|C|}{|G|} \log_{b} (1+S^{-1}).$$
\end{cor}
\begin{proof}
By the effective form of the Chebotarev density theorem due to Lagarias and Odlyzko \cite{C} we know that
\[
\pi_C(x,L/K) = \frac{|C|}{|G|} \Li(x) + o\Big( \frac{x}{(\log x)^2} \Big).
\]
We find via Abel summation that there exist a constant $c_{L/K}\in\mathbb{R}$, depending at most on $L/K$, such that
\[
\sum_{\substack{\mathfrak{p} \in \mathcal{A} \\ \N\mathfrak{p} \leq x}} \frac{1}{\N\mathfrak{p}} = \frac{|C|}{|G|} \log\log x+c_{L/K}+o\Big(\frac{1}{\log x} \Big).
\]
Therefore, $\mathcal{A}$ satisfies (1) and by Theorem 1.1 we obtain
$$ d(\mathcal{A}_{b,S})=\frac{|C|}{|G|}\log_{b}(1+S^{-1}). $$
\end{proof}

We now state some applications of Corollary \ref{cor:1}.

\begin{ex}
Let $a$ and $q$ be coprime integers, $K=\mathbb{Q}$, and $L=\mathbb{Q}(e^{2\pi i/q})$.  We have that  $\mathrm{Gal}(L/K)\cong (\mathbb{Z}/q\mathbb{Z})^{\times}$, and $(\frac{L/K}{p})=\{a\}$ if and only if $p\equiv a \pmod{q}$.  Since $|(\mathbb{Z}/q\mathbb{Z})^{\times}|$ is given by Euler's totient function $\varphi(q):=\prod_{p|q}(1-p^{-1})$, we have that if $\mathcal{A}=\{\textup{$p$ prime}: p\equiv a \pmod{q}\}$, then 
\[
d(\mathcal{A}_{b,S})= \frac{1}{\varphi(q)} \log_{b}(1+S^{-1}).
\]
\end{ex}

\normalsize

\begin{ex}

The following table gives the numerics for the value of the logarithmic density of primes $p \equiv 1 \pmod{5}$ where $p$'s base-$10$ expansion begins with digit $1$. Our expected value, found in the rightmost column, is due to Example 1.3.
\vspace{0.1in}

\begin{center}
\begin{tabular}{|c|c|c|c|c|c||c|}
  \hline
   $x$  & $10^5$ & $5\times 10^5$ & $10\times 10^5$ & $15\times 10^5$ & $2\times 10^6$ & Expected  \\
   \hline
  $d=1$ & $0.0683\dots$ & $0.0705\dots $ & $0.0691\dots $ & $0.0711\dots $ & $0.0724\dots$ & $0.0753\dots$ \\ \hline

\end{tabular}
\end{center}
\vspace{0.1in}
Note that the rate of convergence for this sum generally is extremely slow. The above result is an outlier as it approaches its expected value at a relatively fast rate. Below is a table giving the logarithmic density of primes $p$ whose base-10 expansion begins with the digit $d$. The expected value, found in the rightmost column, is due to Whitney's result. Note in particular that in the case of $d=2$, the value $\frac{1}{2}$ skews our sum significantly. 
\vspace{0.1in}
\begin{center}
        \centering
        \begin{tabular}{|c|c|c|c|c|c||c|}
        \hline
             $x$ & $2\times 10^5$ & $4\times 10^5$ & $6\times 10^5$ & $8\times 10^5$ & $10^6$ & Expected \\
             \hline
$d=1$ & $0.2598\dots$ & $0.2542\dots$ & $0.2511\dots$ & $0.2491\dots$ & $0.2476\dots$ & $0.3010\dots$ \\
\hline
$d=2$ & $0.2931\dots$ & $0.2995\dots$ & $0.2959\dots$ & $0.2935\dots$ & $0.2917\dots$ & $0.1761 \dots$ \\
\hline
$d=3$ & $0.2001\dots$ & $0.2046\dots$ & $0.2022\dots$ & $0.2005\dots$ & $0.1992\dots$ & $0.1249 \dots$ \\
\hline
$d=4$ & $0.0616\dots$ & $0.0603\dots$ & $0.0662\dots$ & $0.0656\dots$ & $0.0652\dots$ & $0.0969 \dots$ \\
\hline
$d=5$ & $0.1194\dots$ & $0.1168\dots$ & $0.1207\dots$ & $0.1197\dots$ & $0.1190\dots$ & $0.0791 \dots$ \\
\hline
$d=6$ & $0.0351\dots$ & $0.0343\dots$ & $0.0339\dots$ & $0.0380\dots$ & $0.0377\dots$ & $0.0669 \dots$ \\
\hline
$d=7$ & $0.0912\dots$ & $0.0893\dots$ & $0.0882\dots$ & $0.0913\dots$ & $0.0907\dots$ & $0.0579 \dots$ \\
\hline
$d=8$ & $0.0256\dots$ & $0.0251\dots$ & $0.0248\dots$ & $0.0246\dots$ & $0.0277\dots$ & $0.0511 \dots$ \\
\hline
$d=9$ & $0.0184\dots$ & $0.0180\dots$ & $0.0178\dots$ & $0.0176\dots$ & $0.0204\dots$ & $0.0457 \dots$ \\
\hline
        \end{tabular}

\end{center}

\end{ex}

\newpage
\begin{ex}
Let $f(x)\in \mathbb{Z}[x]$ be an irreducible polynomial with degree $n$ and Galois group $G$. If $p$ is prime and $f$ is irreducible modulo $p$, then  Dedekind's theorem implies that there is a Frobenius element corresponding to $f$ that is an $n$-cycle.  By the Chebotarev density theorem, the count of such primes $p\leq x$ is $$\frac{1}{|G|}\Li(x)+o\Big(\frac{x}{(\log x)^2} \Big).$$ Therefore, if $\mathcal{A}$ is the set of primes $p$ where $f \pmod{p}$ is irreducible, then by Corollary \cref{cor:1} (with $K=\mathbb{Q}$ and $L$ the splitting field of $f$), we have
\[d(\mathcal{A}_{b,S})= \frac{1}{|G|} \log_{b}( 1 + S^{-1} ). \]
\end{ex}

Finally, we include an example with a positive-density subset of the primes that is provably not defined by Chebotarev conditions but still  satisfies \eqref{1.1}.

\begin{ex}
Let $p_m$ denote the $m$-th prime ordered by absolute value.  Let $r$ and $t$ be integers, $K=\mathbb{Q}$, and $\mathcal{A}$ be the set of primes $p_m$ such that $m \equiv r \pmod{t}$.  Since the count of such primes up to $x$ is $\frac{1}{t}\pi(x)+O(1)$, one verifies \eqref{1.1} by partial summation with $\delta=\frac{1}{t}$.  It follows that
\[d(\mathcal{A}_{b,S})= \frac{1}{t} \log_{b}( 1 + S^{-1}). \]
It was proved in \cite{KO} that $\mathcal{A}$ is not a union of sets of primes determined by Chebotarev conditions.
\end{ex}

\section{Proof of Theorem 1.1}

Let $K$ be a number field, $\mathcal{P}$ be the set of prime ideals of $K$, and $\mathcal{A}$ be a subset of $\mathcal{P}$.  We assume that there exist constants $\delta\geq 0$ and $C\in\mathbb{R}$ such that (\ref{1.1}) holds.  Let $m\geq 1$ and $b\geq 2$ be integers, let $a_1,\ldots,a_m$ be digits in base-$b$, and let $S$ be the string of digits $a_1 a_2 \ldots a_m$.  Let $\mathcal{A}_{b,S}$ be the prime ideals of $\mathcal{A}$ whose norms have base-$b$ expansion starting with $S$.  We prove our result by showing that  $\overline{d(\mathcal{A}_{b,S})}$ and $\underline{d(\mathcal{A}_{b,S})}$ both equal  $\delta \log_{b}( 1 + S^{-1})$.

\begin{proof}
Let $n\geq 1$ and $t\geq 0$ be integers such that $t\leq n$.  By \eqref{1.1}, we know that
$$ \sum_{\substack{Sb^t \leq \N\mathfrak{p} \leq (S+1)b^t \\ \mathfrak{p} \in \mathcal{A}}} \frac{1}{\N\mathfrak{p}} = \begin{cases} 
          C' & t=0 \\
          \delta \log(\frac{\log((S+1)b^t)}{\log(Sb^t)})+o( \frac{1}{t} ) & t \geq 1.
       \end{cases} $$
Here, $C'$ is a positive constant depending on $\mathcal{A}$. Once we sum the above display over $0\leq t\leq n$, it follows that
\begin{align}
    \sum_{\substack{ \N\mathfrak{p} \leq (S+1)b^n \\ \mathfrak{p} \in \mathcal{A}_{b,S}}} \frac{1}{\N\mathfrak{p}} &= \sum_{t=0}^{n}
    \delta \log\Big(\frac{\log((S+1)b^t)}{\log(Sb^t)}\Big)+o\Big(\sum_{t=1}^{n} \frac{1}{t} \Big) \nonumber
    \\ &= \delta \log \Big( \prod_{t=0}^{n}
     \frac{\log_{b}(S+1)+t}{\log_{b}S+t}\Big)+o(\log n).
\end{align}
Notice that
$$\underline{d(\mathcal{A}_{b,S})} \geq \lim_{n \rightarrow \infty} \frac{1}{\log\log(Sb^n)}\sum_{\substack{ \N\mathfrak{p} \leq Sb^{n-m} \\ \mathfrak{p} \in \mathcal{A}_{b,S}}} \frac{1}{\N\mathfrak{p}}  \geq \lim_{n \rightarrow \infty}  \frac{1}{\log\log(Sb^n)} \sum_{\substack{ \N\mathfrak{p} \leq (S+1)b^{n-m-1} \\ \mathfrak{p} \in \mathcal{A}_{b,S}}} \frac{1}{\N\mathfrak{p}} . $$
Combining this inequality with (2), we see that
\begin{equation}
    \underline{d(\mathcal{A}_{b,S})} \geq \lim_{n \rightarrow \infty} \frac{\delta \log \Big( \prod_{t=0}^{n-m-1} \frac{\log_{b}(S+1)+t}{\log_{b}S+t}\Big)+o(\log n)}{\log(\log S+n) +\log\log b}.
\end{equation}
We now recall an identity of Euler for $\Gamma(x)$.  If $x>0$, then
$$\Gamma(x)=\lim_{n \rightarrow \infty} \frac{n! n^x}{\prod_{j=0}^n(x+j)}.$$
Using this identity we find that the product in (2) is asymptotic to a ratio of gamma functions multiplied a power of $n-m-1$; that is, as $n\to\infty$, we have that
\begin{align}
    \prod_{t=0}^{n-m-1} \frac{\log_{b}(S+1)+t}{\log_{b}S+t} &=  \frac{(n-m-1)!(n-m-1)^{\log_b S}}{\prod_{t=0}^{n-m-1}(\log_{b}S+t)}  \frac{\prod_{t=0}^{n-m-1}(\log_{b}(S+1)+t)}{(n-m-1)! (n-m-1)^{\log_b(S+1)}}
    \notag \\
    &\sim \frac{\Gamma(\log_{b} S)}{\Gamma(\log_{b} (S+1))} (n-m-1)^{\log_{b}(S+1)-\log_{b}S}.
\end{align}
Therefore, by (4), we can rewrite (3) as
$$ \underline{d(\mathcal{A}_{b,S})} \geq \lim_{n \rightarrow \infty} \frac{\delta \log \Big( \frac{\Gamma(\log S)}{\Gamma(\log(S+1))} \Big)+\delta\log_b(1+\frac{1}{S})\log(n-m-1)+o(\log n)}{\log(\log S+n) +\log\log b} .$$
By evaluating this limit, we obtain the following lower bound
\begin{equation}
   \underline{d(\mathcal{A}_{b,S})}\geq \delta\log_b\Big(1+\frac{1}{S}\Big).
\end{equation}
We now produce an upper bound for $\overline{d(\mathcal{A}_{b,S})}$. Notice that
$$\overline{d(\mathcal{A}_{b,S})} \leq \lim_{n \rightarrow \infty} \frac{1}{\log \log(Sb^n)} \sum_{\substack{ \N\mathfrak{p} \leq (S+1)b^{n-m} \\ \mathfrak{p} \in \mathcal{A}_{b,S}}} \frac{1}{\N\mathfrak{p}} . $$
By (2),

$$ \overline{d(\mathcal{A}_{b,S})} \leq \lim_{n \rightarrow \infty} \frac{\delta \log \Big( \prod_{t=0}^{n-m} \frac{\log_{b}(S+1)+t}{\log_{b}S+t}\Big)+o(\log n)}{\log(\log S+n) +\log\log b}. $$
By an almost identical argument as above, we observe
$$ \overline{d(\mathcal{A}_{b,S})} \leq \lim_{n \rightarrow \infty} \frac{\delta \log \Big( \frac{\Gamma(\log S)}{\Gamma(\log(S+1))} \Big)+\delta\log_{b}(1+\frac{1}{S})\log(n-m)+o(\log n)}{\log(\log S+n) +\log\log b}. $$
By taking the above limit, we then obtain the following upper bound 
\begin{equation}
\overline{d(\mathcal{A}_{b,S})} \leq \delta \log_b\Big( 1+ \frac{1}{S} \Big).
\end{equation}
Putting together (5) and (6), we deduce the following chain of inequalities 
$$ \delta \log_{b} \Big( 1 + \frac{1}{S} \Big) \leq \underline{d(\mathcal{A}_{b,S})} \leq d(\mathcal{A}_{b,S}) \leq \overline{d(\mathcal{A}_{b,S})} \leq \delta \log_{b} \Big( 1 + \frac{1}{S} \Big),$$
as desired.
\end{proof}

\section*{Acknowledgements}
The author would like to thank Ken Ono for running the
University of Virginia REU in Number Theory and making valuable suggestions to this paper. The author would also like to thank Jesse Thorner and Wei-Lun Tsai for their guidance in research and writing. The author was a participant in the 2022 UVA REU in Number Theory. The author is grateful for the support of grants from the National Science Foundation
(DMS-2002265, DMS-2055118, DMS-2147273), the National Security Agency (H98230-22-1-0020), and the Templeton World Charity Foundation.

\bibliographystyle{alpha}
\bibliography{1}

\begin{thebibliography}{99}
\bibitem{N}Newcomb, S. (1881). \emph{Note on the frequency of use of the different digits in natural numbers.} American Journal of Mathematics, 4(1/4), 39. 
\bibitem{B}Benford, F. (1938). \emph{The Law of Anomalous Numbers.} Proceedings of the American Philosophical Society, 78(4), 551–572. 
\bibitem{W} Whitney, R. E. \emph{Initial digits for the sequence of primes.} Amer. Math. Monthly 79 (1972), 150–152.
\bibitem{C}Lagarias, J. C.; Odlyzko, A. M. \emph{Effective versions of the Chebotarev density theorem.} Algebraic number fields: L-functions and Galois properties (Proc. Sympos., Univ. Durham, Durham, 1975), pp. 409–464. Academic Press, London, 1977.
\bibitem{KO} Kisilevsky, H.; Rubinstein, M. O. \emph{Chebotarev sets.}
Acta Arith. 171 (2015), no. 2, 97–124.

\end{thebibliography}

\end{document}